\documentclass{article}
\pdfoutput=1

\usepackage{amsmath}
\usepackage[utf8]{inputenc}
\usepackage[english]{babel}

\usepackage{amssymb}
\usepackage{amsthm}
\usepackage{faktor}
\usepackage{mathtools}
\usepackage{amsfonts}
\usepackage{enumerate}
\usepackage{enumitem}
\usepackage{multicol}
\usepackage{float}
\usepackage{tikz}
\usepackage{comment}
\usetikzlibrary{arrows,positioning,automata,shadows,fit,shapes}

\usepackage{tikz}
\usetikzlibrary{cd}

\mathtoolsset{showonlyrefs}
\allowdisplaybreaks

\newcommand{\N}{\mathbb{N}}
\newcommand{\Z}{\mathbb{Z}}

\theoremstyle{plain}
\newtheorem{theorem}{Theorem}[section]
\newtheorem{prop}[theorem]{Proposition}

\newtheorem{lemma}[theorem]{Lemma}
\newtheorem{cor}[theorem]{Corollary}
\newtheorem{remark}[theorem]{Remark}

\theoremstyle{definition}
\newtheorem{mydef}[theorem]{Definition}

\newtheorem{notation}[theorem]{Notation}

\author{}

\DeclareMathOperator{\Star}{St}
\DeclareMathOperator{\Link}{Lk}

\newcommand{\ol}{\overline}
\renewcommand{\setminus}{-}

\title{Subgroups of even Artin groups of FC-type}
\author{Yago Antol\'{i}n, Islam Foniqi}

\date{}

\begin{document}
	
\maketitle
\begin{abstract}
We prove a Tits alternative theorem for subgroups of finitely generated even Artin groups of FC type (EAFC groups), stating that there exists a finite index subgroup such that every subgroup of it is either finitely generated abelian, or maps onto a non-abelian free group. 
Parabolic subgroups play a key role, and we show that parabolic subgroups of EAFC groups are closed under taking roots.
\end{abstract}
\vspace{0.3cm}

\renewcommand{\thefootnote}{\fnsymbol{footnote}} 
\renewcommand{\thefootnote}{\arabic{footnote}}

\section{Introduction}
Artin--Tits groups are a class of groups defined through their presentation. 
The family was introduced by J. Tits in the 1960s as generalizations of the presentation of the braid group found by E. Artin.
We recall now the definition to set our convention.
Let $(V,E)$ be a {\it simplicial graph}. i.e. $V$ is a non-empty set and $E$ is a subset of $\{\{u,v\} \mid u,v\in V, u\neq v\}$.
Let $m\colon E\to \{2,3,\dots\}$ be a function.
The pair  $\Gamma = ((V,E),m)$  is called an {\it Artin--Tits system}, and $m$ is called {\it the labeling of $\Gamma$}.

The Artin--Tits group associated to $\Gamma$, denoted by $G_\Gamma$, is the group  
$$G_\Gamma \coloneqq \langle\, V \mid \mathrm{prod}(u,v,m(\{u,v\}))=\mathrm{prod}(v,u,m(\{u,v\})) \,\, \forall \{u,v\}\in E\,\rangle,$$
where $\mathrm{prod}(u,v,n)$ denotes the prefix of length $n$ of the infinite alternating word $uvuvuv\dots$

An Artin--Tits system $\Gamma = ((V,E),m)$ is called {\it even}\, if $m(E)\subseteq 2\mathbb{N}$; it is called {\it right-angled}\, if~$m(E)\subseteq\{2\}$.
The family of right-angled Artin--Tits groups (RAAGs), is the family associated to right-angled Artin--Tits systems. Similarly, the family of even Artin--Tits groups is the one associated to even Artin--Tits systems.

The family of RAAGs stands out as a subfamily of Artin--Tits groups from the richness of its family of subgroups. 
Indeed, from the work of Agol and Wise we know that fundamental groups of hyperbolic 3-manifolds, limit groups, or torsion one-relator groups are (virtually) subgroups of RAAGs.

Even Artin--Tits FC-type groups (EAFC groups)  constitutes a family of Artin--Tits groups that contains the family of RAAGs.
An {\it EAFC system} $\Gamma$ is an even system with the condition that in any triangle of $\Gamma$ there are always at least two edges labeled by 2.
An EAFC group is a group associated to an EAFC system.
It is clear from the definition that all RAAGs are EAFC groups.

There are many features in common between EAFC groups and RAAGs.
Typically, the Artin--Tits groups defined over complete graphs are difficult to understand.
This is not the case for EAFC groups and RAAGs. For instance, the combinatorics of the definition allow to show that if $(\Gamma,m)$ is an EAFC system (resp. right-angled) with $\Gamma$ complete, then $G_\Gamma$ is isomorphic to a direct product of even Artin groups on 1 or 2 generators (resp. 1 generator). 
See Lemma \ref{lem: EAFC over complete graphs}.

Another feature in common between EAFC groups and RAAGs, is that all these groups are poly-free (\cite{HermillerSunic} for RAAGs and \cite{blasco2018poly} for EAFC groups). 
Therefore EAFC groups are locally indicable, and in particular, torsion-free.

In this paper we study subgroups of finitely generated EAFC groups.
Before stating our main result, we recall some terminology. 
A group is called {\it large} if it has a finite index subgroup that maps onto $\mathbb{F}_2$, the non-abelian free group of rank 2.
We use $K$ to denote the fundamental group of a Klein bottle, i.e  $K\cong \langle x,y \mid x^2=y^2\rangle $, which contains $\Z^2$ as an index 2 subgroup.

Our main theorem is the following Tits alternative.

\begin{theorem}\label{thm: main}
Let $G$ be the group associated to the EAFC system $\Gamma$ with $\Gamma$ finite. 
There exists a finite index normal subgroup $G_0$ of $G$ of index at most $\prod_{e\in E} \frac{m(e)}{2}$, such that for any subgroup $H$ of $G$ one of the following holds
\begin{enumerate}
\item[(a)] $H$ is a subgroup of $\Z^s\times K^t$ with $s+2t\leq |V|$ and $H\cap G_0$ is free abelian, or
\item[(b)] $H\cap G_0$ maps onto a non-abelian free group.
\end{enumerate}
In particular, every subgroup of $G$ is either large or virtually abelian.
\end{theorem}

Note that in the case of RAAGs one has $G_0=G$, as $\prod_{e\in E} \frac{m(e)}{2}=1$.
Hence, we have the following corollary.
\begin{cor}[Strongest Tits alternative]\label{cor:strongest}
Any subgroup of a finitely generated RAAG, is either free abelian, or maps onto a non-abelian free group. 
\end{cor}
This corollary was already established by the first author and A. Minasyan in \cite{antolin2015tits}, where the term 'Strongest Tits alternative' was introduced. 
Residually free groups satisfy (trivially) the strongest Tits alternative, 
and RAAGs were the first non-trivial example of groups satisfying such dichotomy.
J. O. Button showed that tubular free by cyclic groups that act freely on a CAT(0) cube complex have a finite index subgroup satisfying the strongest Tits alternative \cite{Button}.
As far as we know, our theorem provides the first new example of groups satisfying the strongest Tits alternative since \cite{antolin2015tits, Button}. 
It is worth mentioning that the method on this paper to find maps onto $\mathbb{F}_2$ differs from  \cite{antolin2015tits}; we explain this in Remark \ref{rem: difference with AM}. 

We remark that $G_0$ is not uniquely defined and depends on certain choices.
In fact, any subgroup as in Definition \ref{defn: G_0} will satisfy the strongest Tits alternative.

One can see that the EAFC group $D_4=\langle a,b\mid abab=baba\rangle$ does not map onto $\mathbb{F}_2$  and $\langle ab,ba\rangle$ is isomorphic to $K$. 
Thus EAFC groups do not satisfy the strongest Tits alternative and thus the necessity of pass to finite index subgroup in  Theorem \ref{thm: main}.

Regarding the quantitative claims of the Theorem,  in case that $\Gamma$ is complete then $G_\Gamma$ does contain a subgroup isomorphic to $\Z^s\times K^t$ with $s+2t=|V|$ (see Lemma \ref{lem: EAFC over complete graphs}). 
So the bound of the 'poly-cyclic rank' (i.e. the Hirsch length) of virtually abelian subgroups is optimal.

Our bound for the index of $G_0$ is optimal for RAAGs.
For the general case, we do not know if the bound is optimal. 
The bound of the index is related to the two-generated case.
An Artin--Tits group over a graph with two vertices is called {\it dihedral} (as the associated Coxeter group is a dihedral group).
Non-free, even dihedral Artin--Tits groups have the following presentation 
$$D_{2n}=\langle a,b \mid (ab)^n = (ba)^n\rangle.$$
An important observation is that $D_{2n}$ contains a normal subgroup $H\cong \mathbb{F}_n\times \mathbb{Z}$ (Lemma \ref{lem: D2n is Fntimes Z by Cn}), and $G/H\cong C_n$, a cyclic group of order $n$. 

This  subgroup $H$ is exactly a the subgroup $G_0$ of the theorem.
Thus for $D_{2n}$ any subgroup has an index $n$ subgroup that is a subgroup of $\mathbb{F}_n\times \mathbb{Z}$, and hence this index $n$ subgroup is either a subgroup of $\Z^2$ or maps onto $\mathbb{F}_2$.

Other Tits alternative style theorems are known for Artin--Tits groups. 
Recall that the original theorem of Tits, that gives the name to the Tits Alternative, says that subgroups of linear groups are either virtually solvable or contain $\mathbb{F}_2$.
From that, and linearity of Artin--Tits groups of spherical type \cite{CohenetAl,Krammer}, we see that spherical Artin--Tits group satisfy this alternative.
For Artin--Tits groups of FC--type, this dichotomy (the original for Tits) was proved by A. Martin and P. Przytycki \cite{MartinPrzytycki2020}. 
For 2-dimensional Artin groups, there are alternatives that restrict the classes of virtually solvable groups that can appear as subgroups \cite{OsajdaPrzytycki2022,MartinPrzytycki2022} to only $\Z,\Z^2$ and $K$ in certain cases.
The full Tits alternative for  2-dimensional Artin groups was settled by A. Martin in \cite{Martin2022}.
None of the mentioned results prove that subgroups containing non-abelian free subgroups are large, besides \cite{antolin2015tits}.

Theorem \ref{thm: main} uses the action of $G_\Gamma$ on a tree with parabolic stabilizers.
Let $\Gamma$ be an Artin--Tits System and $G_\Gamma$ the associated Artin--Tits group.
For $S \subseteq V$, we denote by $G_S$ the subgroup of $G_\Gamma$ generated by the vertices of $S$. 
Subgroups of this form are called {\it standard parabolic subgroups}, and a theorem of Van der Lek \cite{van1983homotopy} (for the whole class of Artin groups) shows that $G_S\cong G_\Delta$ where $(\Delta, m|_{\Delta})$ is the Artin--Tits system induced by $S$.
A subgroup $K$ of $G_\Gamma$ is called {\it parabolic} if it is a conjugate of a standard parabolic subgroup.

The main theorem of our previous paper \cite{antolin2022intersection} says that the intersection of parabolic subgroups of EAFC groups is again parabolic. 
Let $G_{\Gamma}$ be a finitely generated EAFC group and $Y\subseteq G_\Gamma$.
We denote by $P(Y)$ the {\it parabolic closure of $Y$ in $G_\Gamma$}, that is the intersection of all parabolic subgroups of $G_\Gamma$ that contain $Y$.
By \cite{antolin2022intersection}, $P(Y)$ is parabolic.

To show Theorem \ref{thm: main}, not only it is important to understand parabolic closures of elements, but moreover, a key property that we need is the  that parabolic closures are closed under taking roots. 
The second main result of this paper is the following.
\begin{theorem}\label{thm: main roots}
Let $G_\Gamma$ be an EAFC group, $g\in G_\Gamma$ and $n\in \Z-\{0\}$. Then $P(g)=P(g^n)$.
\end{theorem}  

We note that the closure of parabolic subgroups by taking roots is known for other families of groups as spherical type \cite{CumplidoetAl2019}, or large type \cite{CumplidoetAl2023}.

We end the paper giving a characterization of EAFC groups that are coherent. 
This result is not new, as it can be found in the literature. 
But we have not found it with the formulation given here, which serves to support our claim that the family of RAAGs and EAFC groups are intimately related. 
Droms \cite{droms1987graph} showed that RAAGs with defining graph $\Gamma$ are coherent if and only if $\Gamma$ is chordal. 
For EAFC groups we have:
\begin{theorem}\label{thm: coherence in EAFC}
Let $\Gamma$ be an EAFC system. Then $G_\Gamma$ is coherent if and only if $\Gamma$ and $\Gamma^{\leq 2}$ are chordal.
\end{theorem}
Here $\Gamma^{\leq 2}$ denotes the graph obtained from $\Gamma$ by removing edges with label greater than 2.

We have stressed some similarities of  EAFC groups and RAAGs. 
However there are some differences.
It is known that there are  EAFC groups that can  not be virtually co-compactly cubulated.
Indeed, consider the graph $V=\{a,b,c,d\}$ and $E=\{\{a,b\}, \{a,c\}, \{b,d\}, \{c,d\}\}$
with $m(E)=4$.
Then, by \cite[Theorem 4.12.]{Haettel}, $G_\Gamma$ is not virtually co-compactly cubulated.
In particular, $G_\Gamma$ is not commensurable with a RAAG (i.e. $G_\Gamma$ does not have a finite index subgroup isomorphic to a finite index subgroup of a RAAG).
We do not know whether or not every EAFC group has a finite index subgroup that embeds on a RAAG.
In that case, a strongest Tits alternative could have been deduced from the RAAG case (Corollary \ref{cor:strongest}).

The paper is structured as follows. 
In Section \ref{sec: prelim} we collect several basic facts about EAFC groups, including a description of normalizers of parabolic subgroups due to E. Godelle and setting our notation for Bass-Serre theory.
In Section \ref{sec: roots} we collect some facts of our previos work \cite{antolin2022intersection} to prove Theorem \ref{thm: main roots}. In Section \ref{sec: virtually abelian subgroups} we characterize subgroups of EAFC groups not containig free subgroups, and with that, we finish in  Section \ref{sec: large} the proof of our Tits Alternative theorem (Theorem \ref{thm: main}). Finally in Section \ref{sec:coherence}, we sketch a proof for the characterization of coherent EAFC groups.

\section{Preliminaries and notation}
\label{sec: prelim}
In this paper, two types of graphs will be used: simplicial graphs and oriented graphs. 
Simplicial graphs will be only used for Artin--Tits systems and oriented graphs for graphs of groups and graphs with an action, so we will in general use only the term graph, and its nature (simplicial or oriented) should be clear from the context.

Recall that a {\it simplicial graph} $\Gamma=(V,E)$ consists of a non-empty set of vertices $V$ and a set of edges $E$ which is a subset of $\{\{x,y\}\mid x,y\in V, x\neq y\}$.
Given $v\in V$, the {\it link of $v$}, denoted $\Link_\Gamma(v)$, is the set of vertices $u$ of $V$ such that $\{v,u\}\in E$.
The {\it star of $v$}, denoted by $\Star_\Gamma(v)$, is the set $\Link_\Gamma(v)\cup \{v\}$.

Let $\Gamma =((V,E),m)$ be an Artin--Tits system and $G_\Gamma$ be the associated Artin--Tits group; for convenience we will usually denote $G_\Gamma$ just by $G$.
Sometimes, we will just say that $\Gamma$ is an {\it Artin graph} or {\it defining graph of} $G$, and implicitly denote by $m$ or $m_\Gamma$ the labeling function.

For any subset $S\subseteq V$, we denote by $G_S$ the {\it standard parabolic subgroup} $\langle S\rangle.$
As mentioned in the introduction, a theorem of Van der Lek \cite{van1983homotopy}  shows that $G_S\cong G_\Delta$ with the Artin--Tits system~$(\Delta =(V\Delta, E\Delta), m\mid_{E\Delta})$ given by the subgraph $\Delta$ of $\Gamma$ induced by $S$, and $m|_{E\Delta}$ is the restriction of $m$ to $E\Delta$.

For an even Artin--Tits system $\Gamma=((V,E),m)$, i.e. $m(E)\subseteq 2\N$, every standard parabolic subgroup is a retract.
For any $S \subseteq V$ one has a retraction $$\rho_S:G_{\Gamma}\longrightarrow G_S$$ defined on the generators of $G_{\Gamma}$ as:
$\rho_S(s) = s$ for $s\in S$, and $\rho_S(v) = 1$ for $v\in V\Gamma \setminus S$. 

Recall that a {\it parabolic subgroup} $P$ is a conjugate to a standard parabolic subgroup.
It is easy to see that parabolic subgroups are also retracts. 

We recall the main theorem of \cite[Theorem 1.1.]{antolin2022intersection}
\begin{theorem}\label{thm: intersection of parabolics is parabolic}
Let $\Gamma$ be an finite EAFC system. 
The intersection of parabolic subgroups of $G_\Gamma$ is again parabolic.
\end{theorem}

We remark that being parabolic subgroup of an Artin--Tits group is a property of a group and the defining Artin--Tits system. 
We tend to omit the dependency of the Artin--Tits system.
The next lemma guarantees that being parabolic is well behaved under passing to standard parabolic subgroups.

\begin{lemma}\label{lemma: parabolic_inclusion_retractions}
Let $P$ be a parabolic subgroup of the EAFC group $G_{\Gamma}$, and $\Delta$ a subgraph of $\Gamma$. 
Then $P \cap G_{\Delta}$ is parabolic in $G_{\Delta}$.
\end{lemma}

\begin{proof}
Both $P$ and $G_{\Delta}$ are parabolic subgroups of $G_{\Gamma}$, so from Theorem \ref{thm: intersection of parabolics is parabolic}, it follows that $P \cap G_{\Delta}$ is again parabolic in $G_{\Gamma}$. Denote $P \cap G_{\Delta} = g G_Q g^{-1}$ for some $g \in G_{\Gamma}$ and some $Q \subseteq \Gamma$. As $g G_Q g^{-1} = P \cap G_{\Delta} \leqslant G_{\Delta}$ by applying the retraction $\rho_{\Delta}$ one obtains:
$$g G_Q g^{-1} 
=
\rho_{\Delta}(g G_Q g^{-1}) 
=
\rho_{\Delta}(g) \rho_{\Delta}(G_Q) \rho_{\Delta}(g)^{-1} 
= 
\rho_{\Delta}(g) G_{\Delta \cap Q} \rho_{\Delta}(g)^{-1}
= 
h G_{U} h^{-1},$$
for $U = \Delta \cap Q \subseteq \Delta$, and $h = \rho_{\Delta}(g) \in G_{\Delta}$, meaning that $g G_Q g^{-1}$ is parabolic in $G_{\Delta}$, as we wanted.
\end{proof}

Let $(\Gamma,m)$ be a finite EAFC system and $S\subseteq G_{\Gamma}$. 
The  {\it parabolic closure of $S$ in $G_{\Gamma}$}, is the set
$$P^{\Gamma}(S)= \bigcap\limits_{\substack{P \text{ parabolic in }G_\Gamma \\ S \subseteq P}} P.$$

By Theorem \ref{thm: intersection of parabolics is parabolic} for any finitely generated  EAFC group $G_{\Gamma}$, and any set $S \subseteq G_{\Gamma}$, one has that $P^{\Gamma}(S)$ is a parabolic subgroup of $G_{\Gamma}$.

\begin{lemma}\label{lem: conjugation and parabolic closure}
Using the notation above, one has:
\begin{enumerate}
    \item[{\rm (1)}]  $P^{\Gamma}(gSg^{-1})=gP^{\Gamma}(S)g^{-1}$, for any $g\in G_\Gamma$ and any $S\subseteq G_\Gamma$.
    \item[{\rm (2)}] If $\Delta $ is a subgraph of $\Gamma$ and $S\subset G_\Delta$, then  $P^\Delta(S)=P^\Gamma(S)$.
\end{enumerate}
\end{lemma}
\begin{proof} 
(1). As $S \subseteq P^{\Gamma}(S)$, for any $g \in G_{\Gamma}$, one has $g S g^{-1} \subseteq g P^{\Gamma}(S) g^{-1}$. Since $g P^{\Gamma}(S) g^{-1}$ is a parabolic subgroup of $G_\Gamma$ containing $g S g^{-1}$, one obtains $P^{\Gamma}(g S g^{-1}) \subseteq g P^{\Gamma}(S) g^{-1}$. For the other inclusion, we have:
    $$g P^{\Gamma}(S) g^{-1} = g P^{\Gamma}(g^{-1} (g S g^{-1}) g) g^{-1}\subseteq g [ g^{-1} (P^{\Gamma}(g S g^{-1})) g]g^{-1} = P^{\Gamma}(g S g^{-1}).$$

(2). Any parabolic subgroup of $G_\Delta$ is also a parabolic subgroup of $G_\Gamma$, which implies that $P^\Gamma(S) \subseteq P^\Delta(S)$. On the other hand, for any parabolic subgroup $P$ of $G_\Gamma$ containing $S$, the subgroup $P \cap G_\Delta$ is parabolic in $G_\Delta$ which contains~$S$ (see Lemma \ref{lemma: parabolic_inclusion_retractions}). This implies $P^\Delta(S) \subseteq P^\Gamma(S)$ and the equality follows.
\end{proof}

When the ambient defining graph is well understood, we write $P(S)$ instead of $P^\Gamma(S)$ to denote the parabolic closure of $S$ in $G_\Gamma$.

\begin{lemma}\label{lem: normalizer  subsgroup of normalizer of Parabolic Closure}
Let $G_\Gamma$ be an EAFC group. For any $S\subseteq G_\Gamma$, $N_G(S):=\{g\in G \mid g S g^{-1} = S\}\leqslant N_G(P(S))$.
\end{lemma}

\begin{proof}
Let $g \in N_G(S)$, i.e. $gSg^{-1} = S$. 
Taking parabolic closures and applying Lemma \ref{lem: conjugation and parabolic closure} we obtain:
$$P(S) = P(gSg^{-1}) = g P(S) g^{-1},$$
hence $g \in N_G(P(S)).$
\end{proof}

We now explore the structure of EAFC groups based on complete graphs.
We start with the even dihedral Artin--Tits groups.
Recall that we use $\mathbb{F}_n$ to denote the non-abelian free group of rank $n$. 
We also use $C_n$ to denote the cyclic group of order $n$.

The following lemma is well-known.
\begin{lemma}\label{lem: D2n is Fntimes Z by Cn}
The subgroup $N=\langle (ab)^n, a, (ab)a(ab)^{-1},\dots, (ab)^{n-1}a(ab)^{-n+1}\rangle$, of the dihedral Artin group $D_{2n}=\langle a,b \mid (ab)^n = (ba)^n\rangle$, is normal and isomorphic to $\mathbb{F}_n\times \mathbb{Z}$; moreover the quotient of $D_n$ by $N$ is  isomorphic to $C_n$. 
\end{lemma}

\begin{proof}
Set $G = D_{2n}$, and note that $G\cong \langle \, a, \, x \mid x^n = a^{-1}x^na \, \rangle$ via $a \longleftrightarrow a,  ab \longleftrightarrow x$. 

The subgroup $H = \langle\, x^n\,\rangle$ is normal in $G$ as $x^n$ commutes with $a$, and the quotient is equal to:
$$G/H = \langle \, a, \, x \mid x^{n} = 1\,\rangle \cong \langle \, a \, \rangle * \langle \, x \mid x^{n} = 1\,\rangle \cong \Z * C_{n}.$$
Consider $\varphi: G/H\longrightarrow C_{n}$, defined by $\varphi(a) = 0$ and $\varphi(x) = 1$. 
The kernel of $\varphi$ is isomorphic to $\mathbb{F}_n$ with basis $\{x^iax^{-i} \mid 0 \leq i < n\}$. 

Since $H \cong \Z$ is central in $G$ the kernel of the composition $G \longrightarrow G/H\longrightarrow C_{n}$ is equal to $\mathbb{F}_n \times \Z$;
note that this kernel is exactly $N$. 
\end{proof}

The subgroups  of the following definition are the ones that appears in Theorem \ref{thm: main}.
We will show that such subgroups exist immediately.
\begin{mydef}\label{defn: G_0}
Let $\Gamma$ be a finite EAFC system and $G_\Gamma$ the corresponding group.
An {\it appropriate  subgroup}  is a normal subgroup $G_0\leqslant G_\Gamma$ of finite index such that for all 2-generated parabolic subgroup $P$, $P\cap G_0$ satisfies the following alternative: every subgroup of $P\cap G_0$ is either a subgroup of $\mathbb{Z}^2$ or maps onto $\mathbb{F}_2$.
\end{mydef}

\begin{remark}\label{rem: G_0 standard}
If $G_S$ is an standard parabolic of $G_\Gamma$ and $G_0$ is an appropriate subgroup of $G_\Gamma$, then $G_0\cap G_S$ is an appropriate subgroup for $G_S$. Indeed, $G_0\cap G_S$ is finite index and normal in $G_S$. 
Moreover, every 2-generator parabolic  subgroup $P$ of $G_S$ is also a parabolic subgroup of $G_\Gamma$ and henge $P\cap (G_S\cap G_0)=P\cap G_0$ satisfies the required alternative.
\end{remark}

\begin{lemma}\label{lem: G_0}
Let $\Gamma=((V,E),m)$ be a finite EAFC system. 
Then  an appropriate  subgroup $G_0$  of index $[G_\Gamma :G_0]\leq  \prod_{e\in E} \frac{m(e)}{2}$ exists.
\end{lemma}
\begin{proof}
If $E$ is empty, we can take $G_0$ to be $G_\Gamma$.

For each $e=\{a,b\}\in E\Gamma$, fix a finite index subgroup $N_e$ of $G_e$ of index $m(e)/2$ as described in the previous lemma. 
Note that $N_e$ is not uniquely defined and there are two possibilities. 
Then $N_e\cong \mathbb{F}_n\times \mathbb{Z}$ (note that $n$ is allowed to be $1$).
Clearly every subgroup of $N_e$ is either a subgroup of $\mathbb{Z}^2$ or maps onto $\mathbb{F}_2$.
 
For $E$ non-empty, define 
\[
G_0=\cap_{e\in E\Gamma} \rho^{-1}_e(N_e),
\]
where $\rho_e\colon G_\Gamma\to G_e$ is the canonical retraction.
Then $G_0$ is normal in $G_\Gamma$ and  of index $\prod_{e\in E} \frac{m(e)}{2}.$
If $P$ is a parabolic subgroup on two generators, either $P$ is free of rank 2 (and hence every subgroup of $P\cap G_0$ are free) or $P$ is conjugated to some $G_e$, and since $G_0$ is normal, every subgroup of $P\cap G_0$ is isomorphic to a subgroup of $G_e$. 
\end{proof}

\begin{lemma}\label{lem: EAFC over complete graphs}
Let $(\Gamma=(V,E),m)$ be a EAFC system with $\Gamma$ finite and complete. 
The group $G_{\Gamma}$ is a direct product of dihedral Artin--Tits groups and $\mathbb{Z}$'s.
\end{lemma}
\begin{proof}
If  all the edges of $\Gamma$ are labeled with $2$'s then $G_\Gamma$ is isomorphic to the group  $\mathbb{Z}^{|V|}$. 
Next, assume that there is an edge with vertices $a,b$, which is labeled by an even number $k \geq 4$. 
Now, any other vertex $c$ of the graph forms a triangle with $a,b$, because our graph is complete. 
In this triangle the other two edges should be labeled with $2$'s, hence any other generator commutes with $a,b$. 
We have that $G_{\Gamma}=G_{\{a,b\}}\times G_{\Gamma \setminus \{a,b\}}$. 
Now, the subgraph $\Gamma \setminus \{a,b\}$ is still complete, and using an inductive hypothesis we know that the Artin--Tits group based on $\Gamma \backslash \{a,b\}$ is a direct product of even dihedral Artin groups and $\mathbb{Z}$'s; hence our original group is as well.
\end{proof}

We will also need the following well-known fact.
\begin{lemma}\label{lem: D2n is Fn by Z}
The group $D_{2n}=\langle a,b \mid (ab)^n = (ba)^n\rangle$ has a normal subgroup isomorphic to $\mathbb{F}_n$, and the quotient of $D_n$ by this subgroup is isomorphic to $\Z$. 
\end{lemma}
\begin{proof}
Consider the homomorphism $D_{2n}\to \Z$,  $a,b\mapsto 1$. 
Redemeister-Schreier rewriting method gives the infinite presentation of the kernel $\langle a_i, i\in \Z \,|\, a_ia_{i+1}\cdots a_{i+n-1} = a_{i+1}\cdots a_{i+n}, i\in \Z\rangle$. 
It is easy to check that this group is isomorphic to $\mathbb{F}_n$.
\end{proof}

\begin{cor}\label{cor: poly-Z subgroups of D2n}
Every subgroup of $D_{2n}$ not containing a non-abelian free subgroup is isomorphic to a subgroup of $K$.
\end{cor}
\begin{proof}
Let $N\cong \mathbb{F}_n$, be the normal subgroup of $D_{2n}$ of the previous lemma.
Let $H\leqslant D_{2n}$ not containing a non-abelian free subgroup.
If $H\cap N$ is trivial, as $D_{2n}/N\cong \Z$, $H$ is a subgroup of $\Z$ (which is a subgroup of $K$).
If $H\cap N$ is non-trivial, then it must be cyclic. 
Hence $H$ is a subgroup of an extension of $\Z$ by $\Z$. 
Since the only extensions of $\Z$ by $\Z$ are $K$ and $\Z^2$, and $\Z^2$ is a subgroup of $K$, we get our result.
\end{proof}

The Lemma \ref{lem: D2n is Fn by Z} can be generalized to all EAFC groups in the sense that they are all poly-free  \cite{blasco2018poly}.

\subsection{Normalizers of Parabolics}

We will need to use the normalizers of some subgroups of an EAFC group.
Since the normalizer of a subgroup is contained in the normalizer of its parabolic closure (Lemma \ref{lem: normalizer  subsgroup of normalizer of Parabolic Closure}), it will be enough for our applications to understand the normalizer of a parabolic subgroup of an EAFC group.
We collect here a theorem from \cite{godelle2003parabolic} for future use. 
The results in  \cite{godelle2003parabolic} are stated for Artin--Tits groups of FC-type (not necessarily even).
We will not give the definition of Artin--Tits groups of FC-type here, the only thing required to know is that EAFC groups indeed lie in this class.
Using retractions we manage to give more explicit descriptions in the case when the labels are even.

\begin{theorem}[Theorem 0.1, \cite{godelle2003parabolic}]
Let $\Gamma=((V,E),m)$ be an Artin--Tits system of FC-type.
Let $G = G_{\Gamma}$ and $S\subset V$. 
Then one has the equality:
$$N_G(G_S) = G_S \cdot QZ_{G}(G_S).$$
\end{theorem}

The normalizer $N_G(G_S)$, and the quasi-centralizer $QZ_G(G_S)$ respectively, appearing in the theorem above are defined as:
\begin{align*}
    N_G(G_S)  & \coloneqq \{g\in G|gG_Sg^{-1} = G_S\},\\
    QZ_G(G_S) & \coloneqq \{g\in G|gS = Sg\}.
\end{align*}

\begin{lemma}\label{lem: commuting rep}
With the above notation, for an EAFC group $G$, one has 
$$QZ_G(G_S) = \bigcap_{s\in S} QZ_G(G_{\{s\}}).$$
In particular, for every $g\in N_G(G_S)\setminus G_S$ there is $g'\in N_G(G_S)$ with $g'G_S = g G_S$ such that $g'$ commutes with every element of $G_S$.
\end{lemma}

\begin{proof}
Pick $g \in QZ_G(G_S)$ and $s\in S$. By definition $gS = Sg$, or equivalently $gSg^{-1} = S$. Let $s'\in S$ be an element such that $gsg^{-1} = s'$, and apply $\rho_{\{s\}}$. We get $s = \rho_{\{s\}}(s')$, which implies that $s=s'$, and hence $gsg^{-1} = s$. Ultimately $g \in QZ_G(G_{\{s\}})$, and since $s\in S$ was arbitrary we get $QZ_G(G_S) \subset \bigcap_{s\in S} QZ_G(G_{\{s\}})$. The other inclusion is obvious.
\end{proof}

\subsection{Groups acting on trees}

In order to prove our main theorem, we will make use of some Bass-Serre theory.
We collect here some notation and facts for completeness.
We follow \cite{DicksDunwoody}, details and proofs can be found there.

While for Artin--Tits system we used simplicial graphs, when dealing with group actions we will work with oriented graphs.

An {\it oriented graph} $Y=(V,E,\iota,\tau)$ consists of a non-empty set $V$, whose elements are called vertices; a set $E$, whose elements are called edges; and functions $\iota,\tau\colon E\to V$.
We say that the edge $e\in E$ starts at $\iota(e)$ and ends at $\tau(e)$.
There is a natural 1-dimensional CW-complex associated to $Y$, called the topological realization of $Y$, whose 0-skeleton is $V$, its $1$-skeleton is $E\times [0,1]$, and for any $e\in E$, $e\times [0,1]$ is an interval where $(e,0)$ is identified with $\iota e$ and $(e,1)$ is identified with $\tau e$.
The graph is a tree if its topological realization is connected and contractible.

An action of a group $G$ on a graph $Y=(V,E,\iota,\tau)$ is a tuple of actions of $G$ on the sets $V$ and $E$ such that $g\iota(e)=\iota(ge)$ and $g\tau(e)=\tau(ge)$ for all $e\in E$ and all $g\in G$.

A {\it graph of groups} is a pair $(G(-),Y)$ where $Y=(V,E,\overline{\iota},\overline{\tau})$ is an oriented graph and $G(-)$ is a function that to every vertex $v\in V$ assigns a group $G(v)$ and to every edge $e\in E$ assigns a distinguished subgroup $G(e)$ of $G(\overline{\iota}(e))$ and an injective homomorphism $t_e\colon G(e)\to G(\overline{\tau}(e))$, $g\mapsto g^{t_e}$.

Let $Y_0$ be a maximal subtree of $Y$. The {\it fundamental group of a graph of group $(G(-),Y)$ with respect to $Y_0$}, denoted by $\pi_1((G(-),Y,Y_0))$, is the quotient of the free product
$$(*_{v\in V} G(v))* (*_{e\in E} \langle t_e \mid \,\rangle)$$
by the normal closure of the relations
$$t_e^{-1}gt_e=g_e^{t_e},\, e\in E, g\in G(e),\; \text { and } $$
$$t_e =1, \text{ if } e\in EY_0.$$

Main examples of graphs of groups $(G(-),Y)$ are based on graphs with a single edge. 
If $Y=(\{u,v\},\{e\},\iota,\tau)$, then the fundamental group is the free product with amalgamation $G(u)*_{G(e)}G(v)$.
If $Y=(\{v\},\{e\},\iota,\tau)$, then  the fundamental group is the HNN extension $G(v)*_{G(e)}t_e$.

\begin{remark}\label{rem: map onto the fundamental group}
Let $G=\pi_1 ((G(-),Y,Y_0))$ and $N= \langle \cup_{v\in V} G(v)^G\rangle$ the normal closure of the vertex groups in $G$.
Then $G/N\cong \langle \{t_e: e\in EY-EY_0\}\mid \; \rangle$. 
That is, $G/N$ is isomorphic to $\pi_1(Y)$, the fundamental group of the underlying graph of the graph of groups $Y.$
\end{remark}

Given a graph of groups $(G(-),Y)$ with fundamental group $G$, one can construct a   $G$-graph $T$ associated to the graph of groups as follows.
The vertex set $VT=\sqcup_{v\in VY}  G/G(v)$. 
The edge set is $ET=\sqcup_{e\in EY}  G/G(e)$
The adjacency functions are given by $\iota_T(gG(e))= gG(\iota_Ye)$ and $\tau_T (gG(e))=t_egG(\tau_Y e)$.
This graph is a $G$-tree (\cite[Theorem I.7.6]{DicksDunwoody}) and it is called the {\it Bass-Serre tree} of the graph of groups.


Let $G$ be a group action on a tree $T=(VT,ET,\iota_T,\tau_T)$. 
A {\it fundamental $G$-transversal $Y$} consists of subsets $VY\subseteq VT$ and $VE\subseteq ET$ such that $VY$ and $VE$ are $G$-transversals for the action (i.e. they contain exactly one element for each orbit of vertices and edges), $\iota e \in VY$ for all $e\in EY$, and $Y$ is connected (as a subset of the topological realization).

The {\it graph of groups of $G$ acting on $T$ with respect to the fundamental $G$-transversal $Y$} is the graph $G\backslash T$ which is naturally in bijection with $Y$.
Denote by $\ol{\iota}$ and $\ol{\tau}$ the adjacency functions on $G\backslash T$. 
Note that by construction (and using the bijection $Y\leftrightarrow G\backslash T$) one has that $\iota_T(e)=\ol{\iota}(e)$ for all $e\in Y$.
We now describe the function $G(-)$.
For $v\in VY$, we assign $G(v)$ to be $G_v$, the stabilizers of $v\in VT$ by the action of $G$.
For $e\in EY$, we assign $G(e)$ to be $G_e$, the stabilizers of $e\in ET$ by the action of $G$.
Note that since $\iota_T(e)=\ol{\iota}(e)$ for all $e\in Y$, $G(e)$ is naturally a subgroup of $G(\ol{\iota}(e))$.
Finally, let $v=\tau_T(e)$. 
There is some $t_e\in G$ such that $u=t_e^{-1}v\in VY$; we can choose $t_e=1$ if $v\in VY$.
Then $G_e$ is a subgroup of $G_{v}=G_{t_e u}=t_eG_u t_e^{-1}$ and we define the map $G(e)\to G(\ol{\tau} e)=G(u)$ given by $g\mapsto t_e^{-1}gt_e$. 

Bass-Serre theory says that if $G$ acts on a tree $T$ and we construct a graph of groups of the action (with respect to some fundamental $G$-transversal), then the fundamental group of the graph of groups is isomorphic to $G$ and the $G$-tree associated to the graph of groups is $G$-isomorphic to $T$.

\section{Parabolic closures and roots}
\label{sec: roots}

In this section we prove Theorem \ref{thm: main roots}.
The proof is done by induction on the number of edges with labels greater than 2.
A key step in the induction is to pass to certain kernels of retractions.
For that we need to recall the description of such kernels obtained in   \cite{antolin2022intersection}.

\begin{notation}\label{description kernel}
Let $\Gamma=((V,E),m)$ be an EAFC system.
Let $x\in V$, and $\rho\coloneqq \rho_{\{x\}}\colon G_{\Gamma} \to \langle x \rangle$ the associated retraction. 
If $\,\Star(x) = V$, then $K = \ker (\rho_{x})$ is isomorphic to $G_\Omega$, 
where $(\Omega= (V\Omega, E\Omega), m^{\Omega})$ is an EAFC system (Section 4.1 in \cite{antolin2022intersection}) that will be described below.

Moreover, $V \Omega$  comes with an indexing: $i\colon V \Omega \to \Z$.
We will say that $P\leqslant G_\Omega$ is {\it index parabolic} (with respect to $i$) 
if there is $n\in \Z$, $S\subseteq i^{-1}(n)$, and $g\in G_\Omega$, such that $P= gG_Sg^{-1}$.

Let $x$ be a vertex of $\Gamma$ with  $L=\Link(x) = V \setminus \{x\}$.
For $u\in L$, set $k_u= m(\{u,x\})/2$.
Let $\Omega$ be the graph with vertex set 
$$V\Omega = \bigcup_{u\in L}\{u\}\times \{0,1,\dots, k_u-1\}.$$
We define the indexing $i\colon V\Omega \to \Z$ as $i(v,n)=n$. Denote the vertex $(v,n)$ by~$v_n$, and call $v_n\in V\Omega$ {\it of type}~$v$ and of {\it index} $n$.

The edge set of $\Omega$ is 
$$E\Omega = \{\{u_n,v_m\} \mid u_n,v_m \in V\Omega, \{u,v\}\in E\Gamma\}.$$
That is, there is an edge between $u_n$ and $v_m$ in $\Omega$ if and only if there is and edge between $u$ and $v$ in $\Gamma$. 
Finally, the label $m^{\Omega}_{u_n,v_m}$ of $\{u_n, v_m\}$ is the same as the label $m_{u,v}$ of $\{u,v\}$.

\end{notation}

\begin{lemma}[Lemma 4.1 in \cite{antolin2022intersection}]\label{lem: description kernel} With the previous notation, $(\Omega= (V\Omega, E\Omega), m^{\Omega})$ is an even EAFC system and
$G_\Omega \cong \ker (\rho)$ via $v_n\mapsto x^nvx^{-n}$.
\end{lemma}

\begin{remark}
Let $\Gamma$ be an EAFC graph, $v\in V\Gamma$ with $\Star(v) = V$ and suppose that $G_{\Gamma}$ is not a direct product of $G_{v}$ and $G_{V \setminus \{v\}}$. Denote by $E_{>2}\Gamma$ the set of edges of $\Gamma$ with a label greater than $2$, and $G_\Omega$ the Artin--Tits group referred in Lemma \ref{lem: description kernel}. Then $|E_{>2}\Omega| < |E_{>2}\Gamma|$.
\end{remark}

To prove that parabolic subgroups are closed under taking roots, we need to relate parabolic closures on $G_\Gamma$ and $G_\Omega$. 
This is not as well-behaved as when we compared parabolic closures in $G_\Gamma$ and its standard parabolic subgroups (Lemma \ref{lem: conjugation and parabolic closure}), but the control provided by the next lemma will be enough for our purposes.

\begin{lemma}\label{lem: parabolic in Omega}
Let $v\in V$ with $\,\Star(v) = V$, and $K = \ker (\rho_{v}) \cong G_\Omega$. Let $k \in K$ such  $P^{\Gamma}(k)$ is a standard parabolic.
Then, the intersection~$P^\Gamma(k)\cap K$ is parabolic in $K = G_\Omega$. In particular, $P^{\Omega}(k) \leqslant P^\Gamma(k) \cap K$.
\end{lemma}

\begin{proof}
Suppose that  $P^{\Gamma}(k)$ is a standard parabolic, say $P^\Gamma(k) = G_Q$ for some $Q \subset V\Gamma$.  
We are going to show that $G_Q\cap K$ is a standard parabolic in $G_{\Omega}$.

Consider first the case $v \not \in Q$. In this case $G_Q$ is a standard parabolic subgroup of $K$ as well (see Notation \ref{description kernel}) containing $k$, so we have $P^\Gamma(k)\cap K = G_Q$ and moreover, $P^{\Omega}(k) \leqslant G_Q = P^\Gamma(k)$. 

Consider  now the case  $v \in Q$. 
Recall that $K = G_{\Omega}$ with vertices of $\Omega$ of the form $u_i$ as in Notation \ref{description kernel}, where $u \in V \setminus \{v\}$ and $i \in \Z$, with $u_i = v^iu_0v^{-i}$ and $u_0 = u$. 

We claim that $\langle \{u_i \in V\Omega \mid u \in Q \setminus \{v\}, i \in \Z \}\rangle = G_Q \cap K$.
Clearly, $\langle \{u_i  \in V\Omega\mid u \in Q \setminus \{v\}, i \in \Z \} \rangle \leqslant G_Q \cap K$. 
For the other inclusion, pick an element $w\in G_Q \cap K$; since $\rho_v(w)$ is trivial we write $w$ in terms of $u_i = v^iu_0v^{-i}$ for $u \in Q \setminus \{v\}$. 
Finally 
$P^\Gamma(k) \cap K = \langle \{u_i \in V\Omega \mid u \in Q \setminus \{v\}, i \in \Z \}\rangle$ which is a standard parabolic in $K$, and moreover $P^{\Omega}(k) \leqslant P^\Gamma(k) \cap K$.
\end{proof}

\begin{theorem}
For every even FC-type labelled graph, every $g\in G_\Gamma$ and every $n\in \Z \setminus \{0\}$, the equality $P(g) = P(g^n)$ holds.
\end{theorem}

\begin{proof}
Let $E_{>2}\Gamma$ denote the number of edges of $\Gamma$ with a label greater than $2$.
The proof proceeds by induction on $(|E_{>2}|,|V\Gamma|) $ (ordered lexicographically). If $(|E_{>2}|,|V\Gamma|) =(0,1)$, then $|V\Gamma|=1$, and the parabolic closure of any non-trivial element is the whole group $G_\Gamma$ and the lemma holds. Assume now that~$|V\Gamma| > 1$. 

Let $g\in G_\Gamma$ and $n\in \mathbb{Z}\setminus\{0\}.$

If $g$ lies in a proper parabolic subgroup, so does $g^n$. 
In this case there would be a proper subgraph $\Delta$ of~$\Gamma$ and~$t\in G_\Gamma$ such that $tgt^{-1}$ and $tg^nt^{-1} = (tgt^{-1})^n$ lie in $G_\Delta$. 
Notice that $(|E_{>2}\Delta|, |V\Delta|)<(|E_{>2}\Gamma|,|V\Gamma|)$.
By induction $P^\Delta(tgt^{-1})= P^{\Delta}(tg^nt^{-1})$.
As the parabolic closure is well-behaved with conjugation and intersections with restrictions to standard parabolics (Lemma \ref{lem: conjugation and parabolic closure}), $P^\Gamma(g)=P^\Gamma(g^n)$.

So we will assume that $g$ is not contained in any proper parabolic subgroup of $G_\Gamma$. 
In particular $P^{\Gamma}(g)= G_\Gamma$. 
We also let $P^{\Gamma}(g^n)= tG_St^{-1}$ for some subset $S$ of $V$ and $t\in G_\Gamma$. 
We need to show that $S = V$.
As parabolic closure commutes with conjugation (Lemma \ref{lem: conjugation and parabolic closure}), $P(t^{-1}g^nt)=G_S$ and $P(t^{-1}gt)=G_\Gamma$. 
So, replacing $g$ by $t^{-1}gt$ we can assume that $P^{\Gamma}(g^n)=G_S$ and $P^{\Gamma}(g)=G_\Gamma$.

First consider the case where the group $G_{\Gamma}$ can be expressed as a direct product of two proper standard parabolic subgroups as:
$$G_{\Gamma} = G_{\Gamma_1} \times G_{\Gamma_2},$$
then for  any $h = (h_1, h_2)$ we have $P(h) = \langle P^{\Gamma_1}(h_1),  P^{\Gamma_2}(h_2)\rangle \cong P^{\Gamma_1}(h_1)\times P^{\Gamma_2}(h_2)$. In particular, if $g=(g_1,g_2)$ then $P^{\Gamma_i}(g_i)=P^{\Gamma_i}(g_i^n)$ for $i=1,2$ by induction, and hence $P(g)=P(g^n)$.

So assume  that the graph $\Gamma$ is irreducible, i.e. $G_{\Gamma}$ cannot be expressed as a direct product of two proper standard parabolic subgroups. 
If there  are $u, v \in V\Gamma$ with $v \not\in \Star(u)$, then $G_{\Gamma}$ splits as the amalgamated free product: 
$$G_{\Gamma} = G_{\Gamma \setminus \{u\}}*_{G_{\Gamma \setminus \{u, v\}}} G_{\Gamma \setminus \{v\}}.$$
If  $g$ stabilizes a vertex of the Bass-Serre tree corresponding to the splitting above, then $g$ is contained in a parabolic subgroup $P$ over $\Gamma \setminus \{u\}$ or $\Gamma \setminus \{v\}$, contradicting our hypothesis.  So $g$ does not stabilize a vertex, and neither does $g^n$, implying that both $u,v\in S$.

So it remains to show that for each $v\in V$ with $\Star(v)=V$ we have $v\in S$.
We consider two cases, whether $\rho_v(g)$ is trivial or not.

If $\rho_{v}(g)$ is non-trivial, then $\rho_v(g^n)$ is non-trivial. 
As $\rho_{v}(G_S)$ is non-trivial, we have that $v\in S$ as we desired. 


If $\rho_v(g)$ is trivial, then both $g$ and $g^n$ belong to $K = \ker(\rho_v)$ and there is an isomorphism of $K$ with $G_\Omega$ where $\Omega$ is an EAFC system (see Notation \ref{description kernel}). 
We have that $|E_{>2}\Omega|<|E_{>2}\Gamma|$ and hence, by induction we know that for every $h\in K$, $P^\Omega(h)=P^{\Omega}(h^n)$.

Hence $P^\Omega(g)=P^\Omega(g^n)$, and as $P^{\Gamma}(g^n) = G_S$ is a standard parabolic in $G_{\Gamma}$, Lemma \ref{lem: parabolic in Omega}, gives that  $P^\Gamma(g^n)\cap K$ is a parabolic subgroup of $G_\Omega$ containing $g^n$. 
Therefore  $P^{\Omega}(g^n)\leqslant P^{\Gamma}(g^n)$.

Now $g\in P^\Omega(g) = P^\Omega(g^n)\leqslant P^\Gamma(g^n)$, and taking the parabolic closures of these sets we obtain
$$P^\Gamma(g) \leqslant P^\Gamma(P^\Omega(g))= P^\Gamma(P^\Omega(g^n))\leqslant P^\Gamma(P^\Gamma (g^n)) = P^\Gamma(g^n).$$

And therefore, $G_\Gamma= P^{\Gamma}(g)=P^\Gamma(g^n)$, as desired.
\end{proof}

\begin{cor}
Parabolic subgroups are closed by taking roots. That is, if $P$ is parabolic and $g^n\in P$, then $g\in P$. 
\end{cor}

\section{Subgroups not containing non-abelian free groups}
\label{sec: virtually abelian subgroups}
Throughout this section, $\Gamma$ is a finite EAFC system. 
The following proof follows closely the strategy of \cite[Theorem 4.1]{antolin2015tits}.

\begin{theorem}\label{thm: v abelian subgroups}
Every  subgroup of $G=G_\Gamma$, that does not contain a non-abelian free subgroup, is isomorphic to $\mathbb{Z}^s\times K^t$ for some non-negative integers $s,t$ with $s+2t \leq |V|$.

Let $G_0$ be an appropriate subgroup of $G$. Then every  subgroup of $G_0$ (Definition \ref{defn: G_0}), that does not contain a non-abelian free subgroup, is isomorphic to $\mathbb{Z}^s$ with $s \leq |V|$.
\end{theorem}
\begin{proof}
We argue by induction on $|V|.$ 

If $|V|=1$ then $G_\Gamma$ is cyclic, and every subgroup of $G_\Gamma$ is cyclic as well.

For $|V|=2$, the group $G_\Gamma$ is either a free group or a dihedral Artin group. 
In any case (Corollary \ref{cor: poly-Z subgroups of D2n}) subgroups not containing non-abelian free groups must be  subgroups of $K$.
If moreover, we restrict to subgroups of $G_0$ not containing non-abelian free groups, then we are left
 with subgroups of $\Z^2$ (Definition \ref{defn: G_0}).

Assume now that $|V|>2$. We have two cases: the graph is reducible or not.

If $\Gamma$ is reducible, then there are $A, \, B\subseteq V$ such that $G_\Gamma = G_A \times G_B$. 
In this case for any $H \leqslant G_\Gamma$ one has $H \leqslant \rho_A(H) \times \rho_B(H)$. 
If $H\leqslant G_\Gamma$ does not contain a non-abelian free subgroup, then so is the case for $\rho_A(H)$ and $\rho_B(H)$, which by induction are  of the form $\mathbb{Z}^{s_A}\times K^{t_A}$ and $\Z^{s_B}\times K^{t_B}$ respectively, with $s_A+2t_A\leq |A|$ and $s_B+2t_B\leq |B|$. 
Now $H$ is isomorphic to a subgroup of  $$\rho_A(H)\times \rho_B(H) \cong \Z^{s_A+t_B}\times K^{s_A+t_B}$$ which is of the desired form for $s = s_A + s_B$, and $t = t_A + t_B$.
Moreover, one has:
$$s + 2t = (s_A+s_B) + 2 (t_A + t_B) = (s_A + 2t_A) + (s_B + 2t_B) \leq |A| + |B| = |V|.$$
For the case of subgroups of $G_0$, note that $G_0\cong (G_0\cap G_A)\times (G_0\cap G_B)$. By Remark \ref{rem: G_0 standard}, $G_0\cap G_A$ and $G_0\cap G_B$ are appropriate subgroups of $G_A$ and $G_B$ respectively.
Therefore, the same inductive argument shows that if $H$ is a subgroup of $G_0$ not containing a non-abelian free subgroup, then $H\leqslant \Z^s$ with $s\leq |V|$.

Now assume that the graph is irreducible. If $\Link(v) = V \setminus \{v\}$ for any $v \in V$ our graph $\Gamma$ would be complete, and hence reducible as it is EAFC; so there is a vertex $v\in V$ such that $\Link(v)\neq V-\{v\}$. Let $A =\{v\}\cup \Link(v), B = V \setminus \{v\}$ and $C = \Link(v)$. Therefore $G_\Gamma = G_A *_{G_C} G_B$ is a non-trivial amalgamated free product. Let $T$ be the corresponding Bass-Serre tree.

Now a subgroup $H$ of $G_\Gamma$ that does not contain a non-abelian free subgroup, either fixes a vertex of $T$, or there is a bi-infinite line of $T$ invariant under $H$, or $H$ fixes an end of $T$; otherwise, a ping-pong argument allows to construct a non-abelian free subgroup. See, for example \cite[Section 2]{CullerMorgan}.

{\bf Case 1: $H$ fixes a vertex}. Up to conjugation, $H$ is a subgroup of $G_A$ or $G_B$, which by induction implies that $H$ is a isomorphic to $\Z^{s}\times K^t$ where $s + 2t\leq \text{max}\{|A|, |B|\} \leq |V|$. 
Similarly, $H\cap G_0$ is isomorphic to $\Z^s$ with $s\leq \text{max}\{|A|, |B|\} \leq |V|$.

{\bf Case 2: $H$  acts on a bi-infinite geodesic line $\mathcal{L}$}. Here the action of $H$ induces a homomorphism $\Phi:H\rightarrow D_{\infty}$, where $D_{\infty}$ is the infinite dihedral group of all simplicial isometries of $\mathcal{L}$. 
If $\Phi(H)$ is finite, $H$ would fix a vertex of $\mathcal{L}$, and hence we would be on the previous case, so assume that $\Phi(H)$ is infinite. 
If $\Phi(H)\cong D_\infty$, then there is $h\in H$ such that $hG_C\neq G_C$ but $h^2G_C=G_C$. Then $h^2\in G_C$ and by roots closure $h\in G_C$, which is impossible. So this case does not hold.

Therefore, we can assume that $\Phi(H)$ is infinite cyclic.

Put $N\coloneqq \ker \Phi \triangleleft H$; then $N$ fixes every vertex and edge of $\mathcal{L}$ and so it is contained, up to conjugation,  in  $G_C$. 
Therefore, by induction, $N$ is isomorphic to  $\Z^s\times K^t$ with $s + 2t \leq |C|$ (resp. $N\cap G_0$ is isomorphic to $\Z^s$ with $s\leq |C|$). 
We will assume that $N$ is a subgroup of $G_C$, and hence, $P^{\Gamma}(N) \leqslant G_C$.

If $N=\{1\}$ then $H$ is infinite cyclic. 
So, suppose that $N \neq \{1\}$ and notice that $H \leq N_G(N) \leq N_G(P^{\Gamma}(N))$.

As $\Phi(H)$ is infinite cyclic, then $H=N\rtimes \Phi(H)$. 
Let $h\in H$ such that $\Phi(\langle h \rangle)=\Phi(H)$. Note that $h\notin G_C$ as $h$ does not fix this edge. Thus $h\in N_G(P^{\Gamma}(N))\setminus G_C$; as $P^{\Gamma}(N) \leqslant G_C$ we have $h\in N_G(P^{\Gamma}(N))\setminus P^{\Gamma}(N)$ and therefore, by Lemma \ref{lem: commuting rep}, up to changing $h$ with another  element  of $hP^{\Gamma}(N)$, we can assume that $h$  commutes with $P^{\Gamma}(N)$ and with $N$, thus $H \cong N\times \langle h\rangle.$ 
And hence $H$ is isomorphic to  $\Z^s\times K^t$ with $s + 2t \leq |C|+1\leq |V|$ (resp. $H\cap G_0$ is isomorphic to $\Z^s$ with $s\leq |C|+1\leq |V|)$.

{\bf Case 3: $H$ fixes some end $e$ of $\mathcal{T}$.} 
If an element $g\in H$  fixes some vertex $o\in \mathcal{T}$, 
then it will have to fix every vertex of the unique infinite geodesic ray between $o$ and $e$. 
If $h\in H$  is another elliptic element, then it will fix (point-wise) another geodesic ray converging to $e$. 
But any two rays converging to $e$ are eventually the same, in particular they will have a common vertex, which will then be fixed by both $g$ and $h$. Thus $gh\in H$ will also be elliptic. It follows that the subset $N\subset H$, of all elliptic elements of $H$, is a normal subgroup of $H$.
Moreover, as $\Gamma$ is finite, there is a finite subset $X'$ of $N$ such that the parabolic closure of $X'$ and $N$ agree. 
Note that the previous argument shows that there is a common fixed vertex for all elements of $X'$ (and therefore a ray from that vertex to $e$) so the parabolic closure of $X'$ (and hence of  $N$) is contained up to conjugation in $G_C$. 
So, by induction $N$ is isomorphic to  $\Z^s\times K^t$, $s + 2t \leq |C|$ (and $N\cap G_0\cong \Z^s$ with $s\leq |C|)$.

If $H\neq N$, let $x \in H \setminus N$ be an element (necessarily loxodromic) of minimal translation length (for the action of $H$ in $\mathcal{T}$). 
For any other element $y \in H \setminus N$, the intersection of $axis(x)$ and $axis(y)$ is an infinite geodesic ray $\mathcal{R}$, starting at some vertex $p$ of $\mathcal{T}$ and converging to $e$. 
After replacing $x$ and $y$ with their inverses, where necessary, we assume that $x\circ \mathcal{R}\subset \mathcal{R}$, and $y\circ \mathcal{R}\subset \mathcal{R}$. Write $\lVert y\rVert = q \lVert x \rVert + r$ where $q\in \mathbb{N}, r\in \mathbb{N}\cup \{0\}$, and $r < \lVert x \rVert$. 
 Then $y\circ p, (x^{-m}y)\circ p \in \mathcal{R}$ and so $d_{\mathcal{T}}(p, (x^{-m}y)\circ p) = \lVert y \rVert - q \lVert x \rVert = r$. So $\lVert x^{-m}y \rVert \leq r < \lVert x \rVert$, which implies that $x^{-m}y \in  N$ by minimality of $\lVert x \rVert$. 
So $y\in \langle x \rangle  N$ for all $y$ in $H-N$, which means that $H = \langle x \rangle N$.
As $x\in H\leqslant N_G(N)\leqslant N_G(P^{\Gamma}(N))$, we have that $x$ (up to changing it with another $P^{\Gamma}(N)$-coset representative) commutes with $P^{\Gamma}(N)$ and hence with $N$. 
We have that $H\cong N\times \mathbb{Z}$. 
Hence $H$ is isomorphic to $\Z^s\times K^t$, for some $s + 2t \leq |C| + 1$ (and $N\cap G_0\cong \Z^s$ with $s\leq |C|+1)$.
\end{proof}

We need the following fact to prove the main theorem. 
The structure of the proof mimics a bit the previous one, so we will give fewer details.
\begin{lemma}\label{lem: equation in G_0}
Let $\Gamma$ be a finite EAFC system, and $G_\Gamma$ the corresponding group.
Let $G_0$ be an appropriate subgroup.
 The following holds for~$x,y,z\in G_0$:
$$\begin{cases}
xyx^{-1}=y\\
zxz^{-1}=y
\end{cases} \Rightarrow x=y.$$
\end{lemma}
\begin{proof}
The proof is by induction on $|V|$.
If $|V|=1$ there is nothing to prove.
If $|V|=2$, then $\langle x,y,z\rangle$ is either a subgroup of $\mathbb{Z}^2$, in which case there is nothing to prove, or there exist a surjective homomorphism $f\colon \langle x,y,z\rangle \to \mathbb{F}_2$.
In the latter case, we are going to derive a contradiction.
Note that $f(x)$ and $f(y)$ must commute, and hence they must lie in a cyclic subgroup $C$ of $\mathbb{F}_2$. 
If $C$ is trivial, then $f$ can not be surjective, and hence $f(x)$ and $f(y)$ are non-trivial.
Also $f(z)Cf(z)^{-1}\cap C$ is non-trivial, and hence $f(z)$ must lie in a cyclic subgroup with $f(x)$ and $f(y)$, contradicting that $f$ was surjective. 

Assume that $|V|>2$. 
Suppose that $\Gamma$ is reducible, i.e. $V=A\cup B$, $A\cap B=\emptyset$, $G=G_A\times G_B$.
Let $x,y,z$ as in the statement.
Then $\rho_A(x)$ commutes with $\rho_A(y)$ and $\rho_A(z)$ conjugates $\rho_A(x)$ to $\rho_A(y)$, and hence $ \rho_A(x)=\rho_A(y).$
The same occurs with $\rho_B(x),\rho_B(y),\rho_B(z)$, and hence $x=y$.

Let $v\in V\Gamma$ such that $\Link(v)\neq \Delta = \Gamma \setminus \{v\}$. 
We have a splitting $$G_{\Gamma} = G_{\Star(v)} *_{G_{\Link(v)}} G_{\Delta},$$
and let $T$ be the associated Bass-Serre tree. 

Since $x,y$ are conjugated, then either both fix a vertex or both act loxodromically.

If $x$ fixes a vertex $u$ and $y$ fixes $w$ but $x$ does not fix $w$ and $y$ does not fix $u$, then $\langle x,y\rangle$ is non-abelian free. 
A contradiction. 
Thus if $x,y$ fix a vertex, then they fix a common vertex, and hence $x,y\in g G_A g^{-1}$ for a proper parabolic subgroup.
Let $\rho\colon G \to g G_A g^{-1}$ be the retraction associated to $gG_Ag^{-1}$.
Note that $\rho(z)$ conjugates $x$ to $y$.
Thus, by induction $x=y$.

Now suppose that both $x,y$ act loxodromically. Then $\mathrm{axis}(y)=\mathrm{axis}(zxz^{-1})=z\mathrm{axis}(x)$. 
But also as $x,y$ commute, $\mathrm{axis}(x)=\mathrm{axis}(y)$, so we have that $z$ fixes $\mathrm{axis}(x)$.
Let $H=\langle x,y,z\rangle$. 
Then $H$ acts on $\mathcal{L}=\mathrm{axis}(x)$ and we have a homomorphism associated to the action $\Phi\colon H\to \mathrm{Aut}(\mathcal{L})=\mathbb{D}_\infty$.
As $x,y$ act loxodromically, $|\Phi(H)|$ is infinite.
Let $N=\ker \Phi$.
Up to changing $H$ by a conjugate, we can assume that $N\leqslant G_{\Link(v)}$.
As in the previous proof $\Phi(H)\not\cong \mathbb{D}_\infty$ and hence $\Phi(H)$ is infinite cyclic.
Arguing as in the previous proof $H= N\times \langle h\rangle$ for certain $h\in H$.
Let $\rho_N\colon H\to N$ and $\rho_{\langle h\rangle}\colon H\to \langle h\rangle$.
Then $\rho_N(x)$ commutes with $\rho_N(y)$ and $\rho_N(z)$ conjugates $\rho_N(x)$ to $\rho_N(y)$. 
As $N$ is contained in a proper parabolic subgroup, by induction $\rho_N(x)=\rho_N(y)$.
Similarly, $\rho_{\langle h\rangle}(x)=\rho_{\langle h\rangle}(y)$ and hence $x=y$.
\end{proof}


\section{Subgroups containing non-abelian free groups}
\label{sec: large}
In this section we complete the proof of our main theorem. 
We will prove this general criterium
\begin{theorem}\label{thm:criterium}
Let $G$ be a group acting on a tree $T$ with finitely generated abelian stabelizers of rank uniformly bounded.
Suppose that edge stabilizers are direct factors on vertex stabilizers and that for all $x,y,z\in G$ ($xy=yx$ and $zxz^{-1}=x$ implies $x=y$).
If $G$ is not solvable, then $G$ maps onto $\mathbb{F}_2$.
\end{theorem}

The strategy of proof of Theorem \ref{thm:criterium} is quite natural.
In fact, other authors have used similar arguments to show largeness of groups acting on trees with abelian stabilizers. 
For exmple J.O. Button \cite[Theorem 3.7]{Button} for tubular groups and G. Levitt for Generalized Baumslag-Solitar groups  \cite[Theorem 6.7]{Levitt}. 

The proof of Theorem \ref{thm:criterium} is split into two propositions.
The first one is the case where  $G\backslash T$ is not a circle and it is proved in the next proposition. 
This case is  easy when $G\backslash T$ is finite, and slightly more technical in the general case. 

Instead, the second case when   $G\backslash T$ is homeomorhic to a circle is proved in 
Proposition \ref{prop: map to F_2 circle}.


\begin{prop}\label{prop: map to F_2 except from circle}
Let $G$ act on a tree $T$ without a global fixed point. 
Assume that
\begin{enumerate}
\item[(i)] there is no proper $G$-subtree of $T$,
\item[(ii)] there is $M \in \N$ such that vertex stabilizers are free abelian groups of rank $\leq M$,
\item[(iii)] edge stabilizers are direct factors on vertex stabilizers.
\end{enumerate}
If $G\backslash T$ is not homeomorphic to a circle, then $G$ maps onto $\mathbb{F}_2$. 
\end{prop}
\begin{proof}
Let $Y= G\backslash T$ be the quotient graph, and $(G(-),Y)$ a graph of groups associated to the action of $G$ on $T$.

Recall that $G$ maps onto $\pi_1(Y)$ (Remark~\ref{rem: map onto the fundamental group}), so if $\pi_1(Y)$ is non-abelian then we are done. 


So we are restricted to the cases $\pi_1(Y)$ is trivial or infinite cyclic.

Suppose first that $\ol{u}\in VY$ is a valency $1$ vertex and $\ol{e}$ is the edge incident to $\ol{u}$.
Let $u\in VT$ and $e\in ET$ such that $e$ is adjacent to $u$,  $Gu=\ol{u}$, and $Ge=\ol{e}$.

If $\mathrm{rank}(\mathrm{Stab}(u))$ and $\mathrm{rank}(\mathrm{Stab}(e))$ are equal, then since $\mathrm{Stab}(e)$ is a direct factor of $\mathrm{Stab}(u)$, and both are free abelian groups, they must be equal.
This means that  $u$ is a valency 1 vertex in $T$ and clearly, we can remove the orbit of $e$ and $u$ from $T$ and still have a $G$-tree. This however, contradicts (i).
Thus, for every  valency 1 vertex $\ol{u}\in VY$ with adjacent edge  $\ol{e}$,
we have that $G(\ol{e})$ is a proper direct factor of $G(\ol{u})$.

Suppose that there are two valency 1 vertices $\ol{u},\ol{v}\in VY$ with adjacent edges $\ol{e}$ and $\ol{f}$ respectively.
Let  $N$ be the normal subgroup generated by $G(\ol{e}),G(\ol{f})$, and  $G(\ol{z})$ for all $\ol{z}\in VY-\{\ol{u},\ol{v}\}$.
From the presentation of the graph of groups, we see that $G/N=G(\ol{u})/G(\ol{e}) * G(\ol{v})/G(\ol{f})\cong \Z^s*\Z^t$ with $s,t>0$.
Therefore $G$ maps onto $\mathbb{F}_2$.

Suppose that $\pi_1(Y)$ is cyclic  and there is a valency 1 vertex with adjacent edge $\ol{e} = \{\ol{u},\ol{v}\}$.
Similarly as before, let $N$  be the normal subgroup generated by $G(\ol{e})$ and  $G(\ol{z})$ for all $\ol{z}\in VY-\{\ol{u},\ol{v}\}$.
Fix a maximal subtree $Y_0$ of $Y$ and let $\ol{f}$ be the unique edge outside of $Y_0$. 
Note that $\ol{f}\neq\ol{e}$.
From the presentation of the graph of groups, we see that $G/N=G(\ol{u})/G(\ol{e}) * \langle t_{\ol{f}} \rangle\cong \Z^s*\Z$ with $s>0$.
Therefore $G$ maps onto $\mathbb{F}_2$.

So we are restricted to the following cases:
\begin{itemize}[itemsep=0pt, topsep=0pt]
\item[(i)] $\pi_1(Y)$ is infinite cyclic and there is no vertices of valency $1$, or
\item[(ii)] $\pi_1(Y)=\{1\}$ and there is at most one vertex of valency 1.
\end{itemize}
Note that in the first case, as $Y$ is not homemorphic to $\mathbb{S}^1$, $Y$ must consist of an $\mathbb{S}^1$ with trees attached. 
As those trees have no valency one vertices they must be infinite and of infinite diameter.
A similar argument holds of case (ii) so in either case $Y$ must be infinite and of infinite diameter. 

In this setting, there is  $\ol{e}\in EY$ such that  $Y-\{\ol{e}\}$ has at least one infinite connected component, say $Y'$ that is a tree without valency one vertices.
Let $\tilde{Y}'$ be a subtree of $T$ that maps to $Y'$ under the quotient map $T\mapsto G\backslash T$. 
Let $e$ be an edge of $T$ adjacent to $\tilde{Y}'$ mapping to $\ol{e}$.
If $\mathrm{Stab}(e)=\mathrm{Stab(\tilde{Y'})}$ then, removing from $T$ $G$-equivariantly the orbit of $e$ and $\tilde{Y'}$ we get a proper $G$-subtree, contradicting (i).
Therefore $\mathrm{Stab}(e)$ is a proper subgroup of $\mathrm{Stab(\tilde{Y'})}$, which means that there exist an edge $f\in EY'$ such that 
$\mathrm{Stab}(f)$ is a proper subgroup of either $\mathrm{Stab(\iota f)}$ or $\mathrm{Stab(\tau f)}$.
If $\ol{f}$ denotes $Gf$, this means that $G(\ol{f})$ is a proper direct factor either of $G(\iota \ol{f})$ of $G(\tau \ol{f})$.
As $\ol{f}\in EY'$, we have that $Y'-\{\ol{f}\}$ must have an infinite component that is a tree without valency 1 vertices.

Therefore, repeating the previous argument, we have an infinite geodesic path in $Y$ with infinitely many edges $\ol{e}$ such that $G(\ol{e})$ is a proper direct factor either of $G(\iota \ol{e})$ of $G(\tau \ol{e})$.
Now, as ranks of vertex groups are bounded by $M$ there must be two vertices, say $\ol{u},\ol{v}$ with adjacent edges $\ol{e},\ol{f}$ respectively,
such that $\ol{e},\ol{f}$ are in the unique path from $\ol{u},\ol{v}$ and 
 $\mathrm{rank}(G(\ol{e}))<\mathrm{rank}(G(\ol{u}))$ and 
 $\mathrm{rank}(G(\ol{f}))<\mathrm{rank}(G(\ol{v}))$. 
Moreover, as there infinitely many such pairs, we can assume that the connected component of $Y-\{\ol{e},\ol{f}\}$ containing the path between $\ol{u},\ol{v}$ is either empty ($\ol{f}=\ol{e})$ or a tree.

Let $Y_{\ol{u}}$ (resp $Y_{\ol{v}}$) be  the connected component of $Y-\{\ol{e},\ol{f}\}$ containing $\ol{u}$ (resp. $\ol{v}$) and let 
$K_{\ol{u}}$ (resp. $K_{\ol{u}}$) be the fundamental group of the graph of groups $(G(-),Y)$ restricted to $Y_{\ol{u}}$ (resp $Y_{\ol{v}}$).
Finally, if $\ol{e}\neq\ol{f}$ let $K$ be the fundamental group associated to the remaining connected component of $Y-\{\ol{e},\ol{u}\}$.
We see from the presentation of the fundamental group of the graph of groups that 
$$G= K_{\ol{u}}*_{G(\ol{e})}*K*_{G(\ol{f})}*K_{\ol{v}} \quad \text{  if }\ol{f}\neq \ol{e} \qquad \text{ or } \qquad G = K_{\ol{u}}*_{G(\ol{e})}*K_{\ol{v}}\quad  \text{  if }\ol{f}= \ol{e}.$$

Now, if $\pi_1(Y_{\ol{u}})=\mathbb{Z}$, then $K_{\ol{u}}$ maps onto $\Z$ having $G(\ol{e})$ in the kernel (see Remark \ref{rem: map onto the fundamental group}). 
If $\pi_1(Y_{\ol{u}})=\{1\}$, then $Y_{\ol{u}}$ is a tree and the map $\phi\colon: G(\ol{u})\to G(\ol{u})/G(\ol{e})$ extends to a map $K_{\ol{u}}\to  G(\ol{u})/G(\ol{e})$.
Indeed, as $Y_{\ol{u}}$ is a tree and edge groups are direct factors of vertex groups, we can extend $\phi$ to the $G(\ol{z})$ for the $\ol{z}$ that are in the neighborhood of $\ol{u}$. On the direct factor of $G(\ol{z})$ corresponding to the edge group adjacent of $\ol{u}$ and $\ol{z}$ there is only one way to extend $\phi$, but on the direct complement we can extend it freely. 
Repeating this construction inductively, we see that the extension exists.
In either case $K_{\ol{u}}$ maps onto $\Z$ having $G(\ol{e})$ in the kernel, and the same holds for $K_{\ol{v}}$, that is, it maps onto $\Z$ with $G(\ol{f})$ in the kernel.

Thus there are surjective homomorphism from $K_{\ol{u}}$ and $K_{\ol{v}}$ to $\Z$ and we see that they extend to a surjective homomorphism from $G$ to $\mathbb{F}_2$.
\end{proof}

\begin{prop}\label{prop: map to F_2 circle}
Let $G$ be a non-solvable group acting on a tree $T$ with $G\backslash T$ homeomorphic to a circle. 
Assume that
\begin{enumerate}
\item[(i)] vertex stabilizers are free abelian groups, 
\item[(ii)] edge stabilizers are direct factors on vertex stabilizers,
\item[(iii)] for all $x,y,z\in G$, if $xy=yx$ and $zxz^{-1}=y$ then $x=y$.
\end{enumerate}
Then $G$ maps onto $\mathbb{F}_2$. 
\end{prop}
\begin{proof}
Let $(G(-),Y)$ be the graph of groups associated to the action  of $G$ on $T$.
We have that $Y$ is a circle, $H(v)$ is free abelian for all $v\in VY$, for all $e\in EY$ $G(e)$ is a direct factor of $G(\iota e)$ and there is injective homomorphism $t_e\colon G(e)\to G(\tau e)$ such that $t_e(G(e))$ is a direct factor of $G(\tau e)$.

If $|EY|\geq 2$ and $e\in EY$ is such that $G(e)=G(\iota e)$ or $t_e(G(e))=G(\tau(e))$, then we can construct a graph of groups with fewer edges by removing $e$ and properly identifying $G(\iota e)$ with a subgroup of $G(\tau e)$ in the first case, or identifying $G(\tau e)$ with a subgroup of $G(\iota e)$ in the second case.
Observe that the remaining edge groups are still direct factors in the vertex groups in the new graph of groups.

Hence, we will assume that either there is a single edge, or all groups are proper subgroups on the adjancent vertex groups.
We have two cases now.

Case 1: $|EY|=1$. In this case, there is single vertex $u$ of $VY$. 
Here $G(u)$ must be finitely generated abelian.
If $\mathrm{rank}(G(e))=\mathrm{rank}((G(u))$, and taking into account that $G(e)$ and $t_e(G(e))$ are direct factors of $G(u)$, we have that $G(u)=G(e)$ and $t_e\colon G(u)\to G(u)$ is an isomorphism. 
In this case, $G$ is a semidirect product $G(u)\rtimes \Z$ and hence solvable.
A contradiction.

Then  $\mathrm{rank}(G(e))<\mathrm{rank}((G(u))$.
Note that $t_e$ conjugates elements of $G(e)$ to elements of $t_e(G(e))$, but elements of $G(e)$ and $t_e(G(e))$ commute.
By (iii), we have that $t_e\colon G(e)\to G(u)$ is the identity on the subgroup $G(e)$ of $G(u)$.
Then, if $N$ is the normal closure of $G(e)$ in $G$, we get that $G/N=G(u)/G(e)*t$, which maps onto $\mathbb{F}_2$.

Case 2: $|EY|\geq 2$.
Suppose that $VY=\{u_0,\dots u_{n-1}\}$ and $EY=\{e_0,\dots, e_{n-1}\}$ with $\iota e_i=u_i$, $\tau e_i = u_{i+1}$ with the index taken mod $n$.
Consider group $G'$ consisting on the fundamental group of the graph of groups $(G(-),Y')$ obtained by removing $e_n$.
We will show that $\overline{G'}$, the quotient of $G'$ by the normal closure of $G(e_n)\cup t_e(G(e_n))$, has infinite abelianization.
This implies that $G/\langle G(e_n)^G\rangle = \overline{G'} *\langle t_{e_n}\rangle$ maps onto $\mathbb{F}_2$.

The fundamental group of $G'$ is $$G(u_0)*_{G(e_0)} G(u_1)*_{G(e_1)}G(u_2)*\dots *_{G(e_{n-1})}G(u_n).$$
We set $L_i=G(e_i)\leqslant G(u_i)$ and $R_i=t_{e_{i-1}}(G(e_{i-1}))\leqslant G(u_{i})$, indexes mod $n$.
Let $M_i=L_i\cap R_i$.
Note that $M_i$ is a direct factor of $L_i$ and of $R_i$ (and hence of $G(u_i)$), and $\langle L_i,R_i\rangle$ is also a direct factor of $G(u_i)$.
Therefore, using Smith normal forms for finitely generated abelian groups, we can find a base of $G(u_i)$
$$\mu^{(i)}_1,\dots,\mu^{(i)}_{m_i}, \lambda^{(i)}_1,\dots,\lambda^{(i)}_{l_i}, \rho^{(i)}_1,\dots,\rho^{(i)}_{r_1}, \delta^{(i)}_1,\dots,\delta^{(i)}_{d_i}$$
such that $M_i=\langle \mu^{(i)}_1,\dots,\mu^{(i)}_{m_i}\rangle$, $L_i= \langle M_i, \lambda^{(i)}_1,\dots,\lambda^{(i)}_{l_i}\rangle$ and $R_i =   \langle M_i, \rho^{(i)}_1,\dots,\rho^{(i)}_{r_i}\rangle$.
Using these generating sets for each $G(u_i)$ we see that $G'$ is a RAAG on a graph with
$$\sum_{i=0}^n \mathrm{rank}(G(u_i)) - \sum_{i=0}^{n-1}\mathrm{rank}(G(e_i))$$
vertices.

Now we want to show that $G'$ quotient by the normal closure of $R_0\cup L_n$ has infinite abelianization.
Note $R_0$ has basis $\mu^{(0)}_1,\dots, \mu^{(0)}_{m_0},\rho^{(0)}_1,\dots, \rho^{(0)}_{r_0}$ and $L_n$ has basis 
$\mu^{(n)}_1,\dots, \mu^{(n)}_{m_n},\lambda^{(n)}_1,\dots, \lambda^{(n)}_{l_n}$.
Note that $t_{e_n}$  conjugates each member $x$ of the basis of $R_0$ to a member $y$ of the basis of $L_n$.
Therefore, by (iii), if $x$ and $y$ commute, then $x=y$.

To simplify the notation, note that since $G'$ is a RAAG and $\langle G(u_0), G(u_{n})\rangle$ is a parabolic subgroup that contains $R_0\cup L_n$, it is enough to show that if we quotient  $\langle G(u_0), G(u_{n})\rangle$ by the normal closure of $R_0\cup L_n$, the the resulting group  has infinite abelianization.

The situation is as follows. 
The group  $\langle G(u_0), G(u_{n})\rangle$ is the amalgamated product of two finitely generated abelian groups, amalgamated by a free basis.
So we can assume that 
\begin{equation}\label{eq: presentation}
\langle G(u_0), G(u_{n})\rangle= \langle a_1,\dots, a_i, b_1\dots, b_j, c_1,\dots, c_k \,\mid \, c_p \text{ is central},\, [a_p,a_q]=1,\, [b_p,b_q]=1 \, \forall p,q\rangle.
\end{equation}
The  conjugation $t_e$ of members of the basis of $R_0$ to members of the basis of $L_n$ translates to a bijection between generators of the presentation above.
More precisely, let $C=\{c_1,\dots, c_k\}$.
There are proper subsets $A_0\subset \{a_1,\dots, a_i\}\cup C$ and $B_0\subset \{b_1,\dots, b_j\}\cup C$ and a bijection $\pi\colon A_0\to B_0$ such that if $x$ and $\pi(x)$ commute, then $x=\pi(x)$.  
Therefore $\pi(c_s)=c_s$ for all $c_s\in A_0\cap C$ and $\pi(a_s)\in \{b_1,\dots,b_j\}$ for all $a_s\in A_0-C$.
Note that in particular $A_0\cap C= B_0\cap C$.

Let us quotient first $\langle G(u_0), G(u_{n})\rangle$ by then normal subgroup generated by $A_0\cap C$. 
Note that we get a presentation as above, only that we reduce the number of central generators.
Note that still $|A_0- A_0\cap C|< i+k- |A_0\cap C|$ and $|B_0- A_0\cap C|< j+k- |A_0\cap C|$.
So we abuse the notation and we still denote by $A_0$ the set $A_0 -(A_0\cap C)$ and the same for $B_0$, and we still use the same notation as above for $\langle G(u_0), G(u_{n})\rangle/\langle\langle A_0\cap C\rangle\rangle$.
That is, we are assuming we have the presentation \eqref{eq: presentation} with $A_0\subseteq \{a_1,\dots, a_i\}, B_0\subseteq\{b_1,\dots, b_k\}$, generate proper factor of $\langle a_1,\dots,a_i, c_1,\dots, c_k\rangle$ and $\langle b_1,\dots, b_j, c_1,\dots c_k\rangle$ respectively.
This means that $|A_0|< i$ if $k=0$ or $|A_0|\leq i$ if $k>0$, and similarly for $|B_0|$. 
The group $\langle G(u_0), G(u_{n})\rangle/\langle\langle A_0\cup B_0\rangle\rangle$ will have infinite abalianization if  $|A_0\cup B_0|<i+j+k$.
But this holds since if $k=0$, $|A_0\cup B_0|=|A_0|+|B_0|<i+j=i+j+k$ and if $k>0$ then $|A_0\cup B_0|=|A_0|+|B_0|\leq i+j<i+j+k$.
Therefore $G$ maps onto $\mathbb{F}_2$.
\end{proof}

\begin{proof}[Proof of Theorem \ref{thm:criterium}]
Let $G$ act on a tree $T$ as in the hypothesis of the theorem.
As $G$ is non-solvable, it can not have a global fixed point (as it would mean that $G$ is abelian).
By replacing $T$ by a minimal $G$-subtree, we meet condition (i) of Proposition \ref{prop: map to F_2 except from circle}. 
Note that conditions (ii) and (iii) of Proposition \ref{prop: map to F_2 except from circle} follow from the hypothesis, and hence if $G\backslash T$ is not homeomorphic to a circle, it maps onto $\mathbb{F}_2$.
In the case $G\backslash T$ is homeomorphic to a circle, we apply Proposition \ref{prop: map to F_2 circle}.
\end{proof}

\begin{theorem}\label{thm: large}
Let $\Gamma$ be a finite EAFC system.
Let $G_0$ be an appropriate subgroup for $G=G_\Gamma$.
For every  subgroup $H$ of $G$ that is not virtually abelian, $H\cap G_0$ maps onto $\mathbb{F}_2$.
\end{theorem}
\begin{proof}
We argue by induction on $|V|$. 
Let $H$ be a
subgroup of $G_\Gamma$ that is not virtually abelian, and hence, by Theorem \ref{thm: v abelian subgroups} contains a non-abelian free group. 
 In particular, we can assume that $|V|\geq 2$.

If $|V|=2$, we can further assume that $G_{\Gamma}$ is non-abelian. 
If $G_\Gamma$ is free, then so is $H$.
If $G_\Gamma$ is not free, then $G_\Gamma\cong D_{2n}$ with $n>1$ and $G_0\cong \mathbb{F}_n\times \Z$ and hence we can take $H_0=G_0\cap H$ which maps onto $\mathbb{F}_2$.

So assume that $|V|>2$ and that the result holds for EAFC groups based on graphs with fewer vertices.

Suppose that there is proper subset  $A\subseteq V$ such that $\rho_A (H)$ contains a non-abelian free subgroup.
Note that $\rho_A(H\cap G_0)$ has finite index in $\rho_A(H)$ and hence it also contains a non-abelian free subgroup.
As $\rho_A(H \cap G_0)$ is appropiate in $G_A$, it maps onto $\mathbb{F}_2$.
Thus $H\cap G_0$  maps onto $\mathbb{F}_2$.

So, we can assume that for all proper subset $A\subseteq V$, $\rho_A(H)$ does not contain a non-abelian free group.
Thus, we can assume that $(G_0\cap H)\cap G_A\leqslant \rho_A(H)$ does not contain a non-abelian free subgroup.
Therefore, by Theorem \ref{thm: v abelian subgroups}, for all $A\subseteq V$, we can assume that $(G_0\cap H)\cap G_A$ is isomorphic to $\Z^s$ with $s\leq |A|$.

By replacing $H$ with $H\cap G_0$, we will assume that $H$ contains a non-abelian free group and $H\cap gG_Ag^{-1}$ is free abelian on rank $\leq |A|$ for all proper subset $A$ of $V$.

If $\Gamma$ is reducible, there are $A,B\subseteq V$ such that $V=A\cup B$, $A\cap B=\emptyset$,  and $G_\Gamma = G_A\times G_B$. 
Then $H\leqslant\rho_A(H)\times \rho_B(H)\leq \Z^{|A|}\times \Z^{|B|}$ and does not contain a non-abelian free subgroup. 
A contradiction. 
Therefore, we further assume that $\Gamma$ is irreducible.

Let $v\in V\Gamma$ such that $\Link(v)\neq \Delta = \Gamma \setminus \{v\}$. 
We have a splitting $$G_{\Gamma} = G_{\Star(v)} *_{G_{\Link(v)}} G_{\Delta},$$
and let $T$ be the associated Bass-Serre tree. 
Note that $T$ is countable as $G$ is finitely generated.
Now $H$ also acts on $T$ as a subgroup of $G_{\Gamma}$. 
Notice that since $H$ is not contained in a proper parabolic subgroup, $H$ does not fix a vertex.


{\bf Claim.} Let $g\in G$,  $K = H \cap gG_{\Delta}g^{-1}$ or $K= H\cap gG_{\Star(v)}g^{-1}$, and $L = H \cap gG_{\Link(v)}g^{-1}$. 
Then $L$ is closed under taking roots in $K$. 
In particular, $L$ is a direct factor in $K$.

{\it Proof of the Claim.} 
Let $k\in K$ and $n_k>1$ such that $k^{n_k}\in L\subseteq gG_{\Link(v)}g^{-1}$, hence $k \in gG_{\Link(v)}g^{-1}$ because of the closure of roots (Theorem \ref{thm: main roots}). 
Note that $k\in K\subseteq H$, thus $k\in H\cap gG_{\Link(v)}g^{-1}=L$.


By the claim, we have that $H$ acts on $T$ with finitely generated abelian stabilizers of rank $\leq |V|$ and edge stabilizers are direct factors on the vertex stabilizers. 
By Lemma \ref{lem: equation in G_0}, for every $x,y,z\in H$, if $xy=yx$ and $zxz^{-1}=y$ then $y=z$.
As $H$ contains a non-abelian free group, it is not solvable. 
Therefore, by Theorem \ref{thm:criterium}, $H$ maps onto a non-abelian free group.
\end{proof}

\begin{proof}[Proof of Theorem \ref{thm: main}]
Take $G_0$ to be the appropriate subgroup of Lemma \ref{lem: G_0}. It has the desired bound on the index.
Now the theorem follows from Theorem \ref{thm: v abelian subgroups} and Theorem \ref{thm: large}.
\end{proof}

\begin{remark}\label{rem: difference with AM}
The first part of our proof, where we classify the subgroups not containing free groups is very similar to the analogous one in the paper \cite{antolin2015tits}.
However, for the second part, where we prove largeness, the approach is essentially different.

In \cite{antolin2015tits},  the fact that the kernel of any canonical retraction onto 1-generated standard parabolic subgroup is again a RAAG  is used to reduce to the case of subgroups of RAAGs that have non-trivial images under any retractions.
In this situation, and similar to our proof, the main case is when one has a subgroup $H$ of a RAAG $G_\Gamma$ over an irreducible graph $\Gamma$ such that the retraction of $H$ over any 1-generated standard parabolic is non-trivial, and the retraction over any proper parabolic is abelian.
Decompose $G_\Gamma$ as $G_A*_{G_C} G_B$ with $A$ again irreducible.
An abelian subgroup of $H\cap gG_Ag^{-1}$ such that has non-trivial retractions onto 1-generated parabolics of $gG_Ag^{-1}$ can not be contained in a proper parabolic subgroup, and hence it must be cyclic  (as in the proof of Theorem \ref{thm: v abelian subgroups}) and this implies that $H\cap gG_Cg^{-1}$ is trivial and $H$ is a free product of indicable groups.

The analogous key statement (that kernels of retractions onto vertices of RAAGs are  RAAGs again) is not true for EAFC groups. 
Since we did not have that fact, we had to do a finer analysis on the types of graphs of groups obtained from the action of subgroups on splittings of our EAFC; for this, the closure of parabolic subgroups under taking roots and the equations of Lemma \ref{lem: equation in G_0} are fundamental.
These two facts are not used in \cite{antolin2015tits}.
Note that the graphs of groups obtained are not free products as in the RAAG case. 
This is not surprising since this is already happens in the two generator EAFC case, which is the base of induction, and we have subgroups of the form $\Z\times \mathbb{F}_n$ which can be split as a groups of with $\mathbb{Z}^2$ vertex groups and $\mathbb{Z}$ edge groups.
\end{remark}

\section{Coherence}
\label{sec:coherence}
In this section we give an explicit characterization of coherence of EAFC groups that emphasizes its similarities with the class of RAAGs.
Recall that a group $G$ is called {\it coherent}, if any finitely generated subgroup $H$ of $G$, is finitely presented.

A (simplicial) graph $\Gamma =(V,E)$ is said to be {\it chordal}  if for every $S\subset V$ such that the induced subgraph $\Gamma_S$ is a cycle (i.e. homeomorphic to $\mathbb{S}^1$) one has that $|S|=3$.
In other words  all cycles of four or more vertices have a {\it chord}, which is an edge that is not part of the cycle but connects two vertices of the cycle.

Droms provided the following simple characterization of coherent RAAGs

\begin{theorem}[Theorem 1 in \cite{droms1987graph}]
Let $G$ be a RAAG based on $\Gamma$. $G$ is coherent if and only if $\Gamma$ is chordal.
\end{theorem}

Let $\Gamma=((V,E),m)$ be an EAFC system and $\Gamma^{\leq 2}=((V,E'), m|_{E'})$ where $E'=m^{-1}(\{2\})$. 
That is $\Gamma^{\leq 2}$ is obtained from $\Gamma$ by removing all the edges with label $> 2$.

We have the following 

\begin{theorem}[Theorem \ref{thm: coherence in EAFC}]
Let $\Gamma$ be an EAFC system. 
Then $G_\Gamma$ is coherent if and only if both $\Gamma$ and $\Gamma^{\leq 2}$ are chordal.
\end{theorem}

This theorem can be deduced from the classification of coherent Artin groups, which was settled in  \cite{Wise2013} after the reduction in \cite{gordon2004artin}, and it stresses the similarities between EAFC groups and RAAGs. 
Theorem \ref{thm: coherence in EAFC} admits a proof that is very similar to the one of Droms for RAAGs.
We provide here a schematic proof.

We will use the following facts.
\begin{prop}\label{prop: facts coherent}
Let $A,B$ be group and $C$ a common subgroup.
Let $G=A*_CB$. 
Then
\begin{enumerate}
\item[{\rm (1)}] If $A$ and  $B$ are finitely generated, but $C$ is not, then $G$ is not finitely presented.
\item[{\rm (2)}] If $A$ and $B$ are coherent and all subgroups of $C$ are finitely generated, then $G$ is coherent.
\end{enumerate}
\end{prop}
\begin{proof}
(1) is the content of \cite{baumslag1962remark}. (2) is Theorem 8 of \cite{karrass1970subgroups}. 
\end{proof}

\begin{mydef}
Let $\Gamma$ be an Artin system and $G_\Gamma$ the associated group. The map $\phi  \colon G_{\Gamma} \rightarrow \Z, \; v \mapsto 1$ for all $v \in V\Gamma$ defines a morphism, as it respects the relations of $G_{\Gamma}$. 
We define $K_{\Gamma}$ to be the kernel of $\phi$. 
\end{mydef}

%
The follwing is essentially in Droms \cite{droms1987graph}, see also \cite[Lemma 4.4]{HermillerMeier1999}.
\begin{prop}
Let $\Gamma$ be an Artin system, and $U, V$ proper induced subgraphs of $\Gamma$ with $\Gamma = U \cup V$. 
Then 
\[
K_{\Gamma} = K_U *_{K_{U \cap V}} K_V.
\]
\end{prop}

\begin{proof}
As $G_{\Gamma} = G_U *_{G_{U \cap V}} G_V$, the group $G_{\Gamma}$ acts on a Bass-Serre tree $T$, whose vertices are the left cosets of the subgroups $G_U$ and $G_V$ in $G_{\Gamma}$, while the edges are the left cosets of $G_{U \cap V}$ in $G_{\Gamma}$.

The subgroup $K_{\Gamma}$ of $G_{\Gamma}$ acts on $T$ as well; moreover, this action is transitive on edges: 
let $gG_{U \cap V}$ be an edge of $T$, pick $s \in G_{U \cap V}$ such that $\phi(s)=1$.
Then  $s^{\phi(g)}g^{-1} \in K_{\Gamma}$ and one has $$(s^{\phi(g)}g^{-1})(gG_{U \cap V}) = s^{\rho(g)} G_{U \cap V} = G_{U \cap V}.$$
Thus, there is only one orbit of edges under the action of $K_{\Gamma}$. 
The vertices $G_U$ and $G_V$ lie in different orbits of the action of $K_{\Gamma}$ on $T$. 
Finally the vertices stabilizers are of the form $gG_Ug^{-1}\cap K_\Gamma\cong K_U$ or $gG_Vg^{-1}\cap K_\Gamma\cong K_V$ and the edge stabilizers are $gG_{U\cap V}g^{-1}\cap K_\Gamma \cong K_{U\cap V}$.
Ultimately, the quotient $T/K_{\Gamma}$ consists of two vertices joined by an edge, so one has $K_{\Gamma} = K_U *_{K_{U \cap V}} K_V$.
\end{proof}


The following was already observed in \cite[Proposition 4.6]{HermillerMeier1999}.

\begin{prop}\label{thm: graph tree, kernel free}
    If $\Gamma$ is a tree, then $K_{\Gamma}$ is a finitely generated free group.
\end{prop}

\begin{proof}
Expressing $G_\Gamma$ as a graph of groups with dihedral vertex groups and cyclic edge groups, and using the previous proposition, we get that $K_\Gamma$ is a free product of $K_U$ where $U$ is dihedral. 
 Lemma \ref{lem: D2n is Fn by Z} tells us that $K_U$ is a finitely generated free group, and so is $K_{\Gamma}$, as a free product of finitely many such groups.
\end{proof}

\begin{proof}[Proof of Theorem \ref{thm: coherence in EAFC}]
We prove first that the conditions are necessary.

Suppose first that $\Gamma$ is not chordal. 
Then it contains an induced cycle $C$ of length at least $4$. 
Let $x,y$ be two vertices of $C$ that are not adjacent. 
One can express $C$ as a union of two trees $T_1$, $T_2$ with $C = T_1 \cup T_2$ and $T_1 \cap T_2 = \{x, y\}$. 
From this, one has a splitting:
\[
    K_{C} = K_{T_1} * _{K_{\{x,y\}}} K_{T_2}.
\]
From Theorem \ref{thm: graph tree, kernel free}, both $K_{T_1}$, and $K_{T_2}$ are finitely generated free groups, so $K_C$ is finitely generated as well. The group $K_{\{x,y\}}$ is the normal closure of $xy^{-1}$ in the free group $G_{\{x,y\}}$, so $K_{\{x,y\}}$ is not finitely generated, and hence by Proposition \ref{prop: facts coherent}, $K_C$ is not finitely presented. 
This means that $G_C$ is not coherent, and hence $G_{\Gamma}$ is not coherent as well.

It remains to consider the case when $\Gamma$ is chordal but $\Gamma^{\leq 2}$ is not.
Then $\Gamma^{\leq 2}$ contains an induced cycle $C$ of length at least $4$. 
Suppose that the length of $C$ is  $>4$, then there are two chords in $\Gamma$ with label $> 2$ and an edge of $C$ forming a triangle, contradicting that $\Gamma$ is EAFC.
So, assume then that there is an induced cycle $C$ of length 4 in $\Gamma^{\leq 2}$. 
This means that $G_C$, the standard parabolic of $G_\Gamma$ spanned by $C$, is either isomorphic to $D_{2s}\times D_{2t}$ with $s,t>1$ or isomorphic to $(\langle a \rangle \times D_{2n}) *_{D_{2n}}(\langle a \rangle \times D_{2n})$ with $n>1$. 
As $D_{2n}$ contains $\mathbb{F}_n$, we see that in both cases $G_C$ contains $\mathbb{F}_2\times \mathbb{F}_2$, which is not coherent. 

We now prove that the conditions are sufficient.

As coherence is preserved under free products, it is enough to consider the case when $\Gamma$ is connected and it satisfies the hypothesis of the theorem.

If $\Gamma$ is complete, then by the hypothesis of the theorem it can contain at most one edge with label greater than $2$. By Lemma \ref{lem: EAFC over complete graphs}, $G_{\Gamma}$ is a direct product of a dihedral Artin group and some copies of $\Z$, so it is coherent.

If $\Gamma$ is not complete, then there are two proper subgraphs $\Gamma_1$ and $\Gamma_2$ with $\Gamma = \Gamma_1 \cup \Gamma_2$, and $\Gamma_1 \cap \Gamma_2 = A$, with $A$ complete. Here one has a splitting:
\[
G_{\Gamma} = G_{\Gamma_1} * _{G_{A}} G_{\Gamma_2}.
\]
One important thing to notice here is that $G_A$ is abelian. Indeed, if there is an edge $\{a,b\} \subset A$ with label $m > 2$, then from the construction, one can find $x \in \Gamma_1$, and $y \in \Gamma_2$ where $a, x, b, y$ form a square with a chord labeled by $m$; but this is not allowed by the hypothesis of the theorem. Now the result follows from Propostion \ref{prop: facts coherent}.
\end{proof}

\vspace{1cm}

\noindent{\textbf{{Acknowledgments}}} 
We thank Thomas Haettel for explaining his results on cubulating Artin--Tits groups.
Also to A. Minasyan for pointing out \cite{Button} and J.O. Button for discussions about his work.

Yago Antol\'{i}n  acknowledges partial support from the Spanish Government through the ``Severo Ochoa Programme for Centres of Excellence in R\&{}D'' CEX2019-000904-S, and the grant PID2021-126254NB-I00 funded by MCIN/ AEI /10.13039/501100011033. 
 He also acknowledges partial support from the project ``Santander-UCM 2021'' PR44/21-29907.

Islam Foniqi acknowledges past support from the Department
of Mathematics of the University of Milano-Bicocca, the Erasmus Traineeship grant 2020-1-IT02-KA103-078077, and current support from the EPSRC Fellowship grant EP/V032003/1 ‘Algorithmic, topological and geometric aspects of infinite groups, monoids and inverse semigroups’.
\bibliography{main}
\bibliographystyle{plain}

\noindent\textit{\\ Yago Antol\'{i}n,\\
Fac. Matem\'{a}ticas, Universidad Complutense de Madrid and \\ 
Instituto de Ciencias Matem\'aticas, CSIC-UAM-UC3M-UCM\\
Madrid, Spain\\}
{email: yago.anpi@gmail.com}

\noindent\textit{\\ Islam Foniqi,\\
School of Mathematics, University of East Anglia,\\
 Norwich NR4 7TJ, \\
 England, UK\\}
{email: I.Foniqi@uea.ac.uk}
\end{document}